\documentclass[11pt,twoside, leqno]{article}

\usepackage{mathrsfs}
\usepackage{amssymb}
\usepackage{amsmath}
\usepackage{amsthm}
\usepackage{amsfonts}
\usepackage{color}

\allowdisplaybreaks

\def\rr{{\mathbb R}}
\def\rn{{{\rr}^n}}
\def\zz{{\mathbb Z}}

\def\nn{{\mathbb N}}
\def\fz{\infty}

\def\aa{{\mathbb A }}

\def\cs{{\mathcal S}}

\def\cb{{\mathcal B}}

\def\az{\alpha}

\renewcommand\tilde{\widetilde}

\def\esup{\mathop\mathrm{\,ess\,sup\,}}

\def\ls{\lesssim}
\def\lz{\lambda}
\def\Lz{\Lambda}

\def\bz{\beta}
\def\vz{\varphi}

\def\Oz{{\Omega}}

\def\hs{\hspace{0.3cm}}

\def\r{\right}
\def\lf{\left}

\def\bint{{\ifinner\rlap{\bf\kern.30em--}
\int\else\rlap{\bf\kern.35em--}\int\fi}\ignorespaces}

\def\sbint{{\ifinner\rlap{\bf\kern.32em--}
\hspace{0.078cm}\int\else\rlap{\bf\kern.45em--}\int\fi}\ignorespaces}

\def\dsup{\displaystyle\sup}

\newtheorem{theorem}{Theorem}[section]
\newtheorem{lemma}[theorem]{Lemma}

\theoremstyle{definition}
\newtheorem{remark}[theorem]{Remark}
\newtheorem{definition}[theorem]{Definition}
\numberwithin{equation}{section}

\numberwithin{equation}{section}

\numberwithin{equation}{section}

\begin{document}

\arraycolsep=1pt

\title{\Large\bf Weighted Norm Inequalities for Parametric Littlewood-Paley Operators
\footnotetext{\hspace{-0.35cm}
{\it 2010 Mathematics Subject Classification}.
{Primary 42B25; Secondary 42B30, 46E30.}
\endgraf{\it Key words and phrases.} Littlewood-Paley operator, Hardy space, Muckenhoupt weight,
Musielak-Orlicz function.
}}
\author{ Li Bo\\
}
\date{  }
\maketitle

\vspace{-0.8cm}

\begin{minipage}{13cm}\small
{
\noindent
{\bf Abstract:}
In this paper,
the author establishes the boundedness of parametric Littlewood-Paley operators
from Musielak-Orlicz Hardy space to Musielak-Orlicz space, or to weak Musielak-Orlicz space at the critical index.
Part of these results are new even for classical Hardy space of Fefferman and Stein.
}
\end{minipage}



\section{Introduction\label{s1}}
The impact of the theory of Hardy space $H^p(\rn)$ with $p\in(0,\,1]$ in the last forty years has been significant.
Hardy space first appeared in the work of Hardy \cite{h14} in 1914. Its study was
based on complex methods and its theory was one-dimensional.
The higher dimensional Euclidean theory of Hardy space was developed by Fefferman
and Stein \cite{fs71} who proved a variety of characterizations for them.
Later, the advent of its atomic or molecular characterizations
enabled the extension of $H^p(\rn)$ to far more general settings such as space
of homogeneous type in the sense of Coifman and Weiss \cite{cw77}.
It is well known that, when $p\in(0,\,1]$, Hardy space $H^p(\mathbb{R}^n)$ is a good substitate of the Lebesgue space $L^p(\mathbb{R}^n)$
in the study for the boundedness of operators. For example, when $p\in(0,\,1]$,
the Riesz transforms are not bounded on $L^p(\mathbb{R}^n)$, however, they are bounded on $H^p(\mathbb{R}^n)$.

Recently, Ky \cite{k14} introduced a new Musielak-Orlicz Hardy space $H^\vz(\rn)$,
which unifies the classical Hardy space, the weighted Hardy space, the Orlicz Hardy space, and
the weighted Orlicz Hardy space, in which the spatial and the time variables may not be separable (see, for example, \cite{fs72,st89,jy10,lhy12}).
Apart from interesting theoretical considerations, the motivation to study
$H^\vz(\rn)$ comes from applications to elasticity,
fluid dynamics, image processing, nonlinear PDEs and the calculus of variation (see, for example, \cite{d05,dhr09}).
More Musielak-Orlicz-type spaces are referred to
\cite{bckyy13,ccyy17,lfy15,lffy16,lsl16,ly13,lyj16,ylk17,yyz14,zql17}.

On the other hand, various fields of analysis and differential equations require the theory of various function
space, for examples, Lebesgue space, Hardy space, various forms of Lipschitz space, BMO space and Sobolev space.
From the original definitions of these spaces, it may not appear that they are very closely related.
There exist, however, various unified approaches to their study.
The Littlewood-Paley theory, which arises naturally from the consideration of the Dirichlet problem, provides one of the
most successful unifying perspectives on these function spaces (see \cite{fjw91} for more details).
And, it remains closely related to the theory of Fourier multipliers (see \cite[Chapter 5]{g09c}).

Suppose that $S^{n-1}$ is the unit sphere in the $n$-dimensional Euclidean space $\rn \ (n\ge2)$.
Let $\Omega$ be a {homogeneous function of degree zero} on $\rn$ which is locally integrable and satisfies the cancellation condition
\begin{align}\label{e1.1}
\int_{S^{n-1}}\Omega(x')\,d\sigma(x')=0,
\end{align}
where $d\sigma$ is the Lebesgue measure and $x':=x/{|x|}$ for any $x\neq{\mathbf{0}}$.
For a function $f$ on $\rn$, the parametric Littlewood-Paley operators
$\mu^\rho_{\Omega,\,S}$ and $\mu^{\rho,\,\ast}_{\Omega,\,\lz}$ are, respectively,
defined by setting, for any $x\in\rn$,
$$\mu^\rho_{\Omega,\,S}(f)(x)£º=\lf(\int\int_{|y-x|<t}\lf|
\int_{|y-z|<t}\frac{\Omega(y-z)}{|y-z|^{n-\rho}}f(z)\,dz\r|^2\,\frac{dydt}{t^{n+2\rho+1}}\r)^{\frac12}$$
and
$$\mu^{\rho,\,\ast}_{\Omega,\,\lz}(f)(x)
=\lf[\int\int_{{\mathbb{R}}^{n+1}_+}\lf(\frac{t}{t+|x-y|}\r)^{\lz n}\lf|
\int_{|y-z|<t}\frac{\Omega(y-z)}{|y-z|^{n-\rho}}f(z)\,dz\r|^2\,\frac{dydt}{t^{n+2\rho+1}}\r]^{\frac12},$$
where $\rho\in(0,\,\infty)$ and $\lz\in(1,\,\fz)$. The $\mu^\rho_{\Omega,\,S}$ and $\mu^{\rho,\,\ast}_{\Omega,\,\lz}$
were first studied by Sakamoto and Yabuta \cite{sy99} in 1999. They showed that
if $\Omega\in{\rm{Lip}}_\alpha(S^{n-1})$ with $\alpha\in(0,\,1]$,
then $\mu^\rho_{\Omega,\,S}$ and $\mu^{\rho,\,\ast}_{\Omega,\,\lz}$ are bounded on $L^p(\rn)$ with $p\in(1,\,\fz)$.
In 2009, Xue and Ding \cite{xd07} obtained a celebrated result that
$\mu^\rho_{\Omega,\,S}$ and $\mu^{\rho,\,\ast}_{\Omega,\,\lz}$ are bounded on $L^p_\omega(\rn)$ with $p\in(1,\,\fz)$
under weaker smoothness condition of $\Omega$,
where $\omega\in A_p$ and $A_p$ denotes the Muckenhoupt weight class.
As for their Hardy space boundedness, Ding et al. \cite{dlx07a,dlx07lp}
showed that, if $\Omega$ satisfies some weaker smoothness condition,
then $\mu^\rho_{\Omega,\,S}$ and $\mu^{\rho,\,\ast}_{\Omega,\,\lz}$ are bounded from $H^1(\rn)$ to $L^1(\rn)$.
More conclusions of Littlewood-Paley operators are referred to \cite{agd14,cxy15,hxmy15,l17,lw14,lwz15,lfy15,lll17,sxy14,xyy15}.

Motivated by all of the above mentioned facts, a natural and interesting problem arises, that is to say,
whether $\mu^\rho_{\Omega,\,S}$ and $\mu^{\rho,\,\ast}_{\Omega,\,\lz}$ are bounded
from Musielak-Orlicz Hardy space $H^\vz(\rn)$ to Musielak-Orlicz space $L^\vz(\rn)$.
In this paper we shall answer this problem affirmatively. Not only that,
we also discuss boundedness of $\mu^\rho_{\Omega,\,S}$ and $\mu^{\rho,\,\ast}_{\Omega,\,\lz}$
from Musielak-Orlicz Hardy space $H^\vz(\rn)$ to weak Musielak-Orlicz space $WL^\vz(\rn)$ at the critical index.

The present paper is built up as follows.
In Section \ref{s2}, we recall some notions concerning Muckenhoupt weights,
growth functions and Musielak-Orlicz Hardy space $H^\vz(\rn)$.
Then we statement the boundedness of $\mu^\rho_{\Omega,\,S}$ and $\mu^{\rho,\,\ast}_{\Omega,\,\lz}$
from $H^\vz(\rn)$ to $L^\vz(\rn)$ or to $WL^\vz(\rn)$
(see Theorems \ref{dingli.1}-\ref{dingli.4} below),
the proofs of which are given in Sections \ref{s3} and \ref{s4}.
In the process of the proofs of main results, a boundedness criterion of operators
from $H^\vz(\rn)$ to $L^\vz(\rn)$ (see \cite[Lemma 3.12]{lll17}) plays an indispensable role.
Motivated by this, a boundedness criterion of operators
from $H^\vz(\rn)$ to $WL^\vz(\rn)$ (see Lemma \ref{yt2} below) is also established.

Finally, we make some conventions on notation.
Let $\zz_+:=\{1,\, 2,\,\ldots\}$ and $\nn:=\{0\}\cup\zz_+$.
For any $\bz:=(\bz_1,\ldots,\bz_n)\in\nn^n$,
let $|\bz|:=\bz_1+\cdots+\bz_n$.
Throughout this paper the letter $C$ will denote a \emph{positive constant} that may vary
from line to line but will remain independent of the main variables.
The \emph{symbol} $P\ls Q$ stands for the inequality $P\le CQ$. If $P\ls Q\ls P$, we then write $P\sim Q$.
For any sets $E,\,F \subset \rn$, we use $E^\complement$ to denote the set $\rn\setminus E$,
$|E|$ its  {\it{$n$-dimensional Lebesgue measure}},
$\chi_E$ its \emph{characteristic function} and
$E+F$ the {\it algebraic sum} $\{x+y:\ x\in E,\,y\in F\}$.
For any $s\in\rr$, $\lfloor s\rfloor$ denotes the
unique integer such that $s-1<\lfloor s\rfloor\le s$.
If there are no special instructions, any space $\mathcal{X}(\rn)$ is denoted simply by $\mathcal{X}$. For instance, $L^2(\rn)$ is simply denoted by $L^2$.
For any index $q\in[1,\,\fz]$, $q'$ denotes the {\it{conjugate index}} of $q$, namely, $1/q+1/{q'}=1$.
For any set $E$ of $\rn$, $t\in[0,\,\infty)$ and measurable function $\vz$,
let $\vz(E,\,t):=\int_E\vz(x,\,t)\,dx$ and $\{|f|>t\}:=\{x\in\rn: \ |f(x)|>t\}$.
As usual, for any $x\in\rn$, $r\in(0,\,\infty)$ and $\alpha\in(0,\,\infty)$, let $B(x,\,r):=\{y\in\rn: |x-y|<r\}$ and $\alpha B(x,\,r):=B(x,\,\alpha r)$.




\section{Notions and main results}\label{s2}
In this section, we first recall the notion concerning the Musielak-Orlicz Hardy space $H^\vz$
via the non-tangential grand maximal function, and then present
the boundedness of parametric Littlewood-Paley operators
from Musielak-Orlicz Hardy space to Musielak-Orlicz space, or to weak Musielak-Orlicz space at the critical index..

Recall that a nonnegative function $\vz$ on $\rn\times[0,\,\fz)$ is called a {\it Musielak-Orlicz function} if,
for any $x\in\rn$, $\vz(x,\,\cdot)$ is an Orlicz function on $[0,\,\fz)$ and, for any $t\in[0,\,\fz)$,  $\vz(\cdot\,,t)$ is measurable on $\rn$.
Here a function $\phi: [0,\,\fz) \to [0,\,\fz)$ is called an {\it Orlicz function},
if it is nondecreasing, $\phi(0) = 0$, $\phi(t) > 0$
for any $t\in(0,\,\fz)$, and $\lim_{t\to\fz} \phi(t) = \fz$.


Given a Musielak-Orlicz function $\vz$ on $\rn\times[0,\,\fz)$,
$\vz$ is said to be of {\it{uniformly lower}} (resp. {\it{upper}}) {\it{type}} $p$ with $p\in \mathbb{R}$,
if there exists a positive constant $C:=C_{\vz}$ such that, for any $x\in\rn$, $t\in[0,\,\fz)$ and $s\in(0,\,1]$
(resp. $s\in[1,\,\fz)$),
\begin{eqnarray*}
\vz(x,\,st)\le C s^p\vz(x,\,t).
\end{eqnarray*}
The {\it critical uniformly lower type index}
and the {\it critical uniformly upper type index}
of $\vz$
are, respectively,
defined by
\begin{align}\label{e2.1}
i(\vz):=\sup\{ p\in\mathbb{R}: \ \vz \mathrm{ \ is \ of \ uniformly\  lower\  type \ {\it p}} \},
\end{align}
and
\begin{align}\label{e2.1.1}
I(\vz):=\inf\{p\in\mathbb{R}: \ \vz \mathrm{ \ is \ of \ uniformly\  upper\  type \ {\it p}} \}.
\end{align}
Observe that $i(\vz)$ or $I(\vz)$ may not be attainable,
namely, $\vz$ may not be of uniformly lower type $i(\vz)$ or
of uniformly upper type $I(\vz)$ (see \cite[p.\,415]{lhy12} for more details).

\begin{definition}\label{d2.2}
Let $q\in[1,\,\fz)$. A locally integrable function $\vz(\cdot\,,t): \rn \to [0,\,\fz)$ is said to satisfy
the {\it uniformly Muckenhoupt condition} $\aa_q$,
denoted by $\vz\in \aa_q$, if there exists a positive constant $C$ such that,
for any ball $B\subset\rn$ and $t\in(0,\,\fz)$, when $q=1$,
$$\frac{1}{|B|}\int_{B} \vz(x,\,t)\,dx\lf\{\esup_{y\in B} [\vz(y,\,t)]^{-1}\r\}\le C$$
and, when $q\in(1,\fz)$,
$$\frac{1}{|B|}\int_{B}\vz(x,\,t)\,dx
\lf\{\frac{1}{|B|}\int_{B}[\vz(y,\,t)]^{-\frac{1}{q-1}}\,dy\r\}^{q-1}
\le C.$$

\end{definition}

For $\vz\in\aa_q$ with $q\in[1,\,\fz)$, we have the following properties as the classical Muckenhoupt weight.
\begin{lemma}\label{bcd}{\rm{\cite[Lemma 4.5]{k14}}}
Let $\vz\in\mathbb{A}_q$ with $q\in[1,\,\fz)$. Then the following statements hold true:
\begin{enumerate}
\item[\rm{(i)}]
there exists a positive constant $C$ such that,
for any ball $B\subset\rn$, $\lambda\in(1,\,\fz)$ and $t\in(0,\,\fz)$,
$$\vz(\lambda B,\,t)\le C{\lambda}^{nq}\vz(B,\,t).$$

\item[\rm{(ii)}]
if $q\neq1$, there exists a positive constant $C$ such that,
for any ball $B(x_0,\,r)\subset\rn$ and $t\in(0,\,\fz)$,
$$\int_{B^\complement}\frac{\vz(x,\,t)}{|x-x_0|^{nq}}\,dx\le C\frac{\vz(B(x_0,\,r),\,t)}{r^{nq}}.$$
\end{enumerate}
\end{lemma}

Define $\aa_\fz:=\bigcup_{q\in[1,\,\fz)} \aa_q$.
It is well-known that if $\vz\in \aa_q$ with $q\in[1,\,\fz]$,
then $\vz^\varepsilon\in \aa_q$ for any $\varepsilon\in(0,\,1]$
and $\vz^\eta\in \aa_q$ for some $\eta\in(1,\,\fz)$.
Also, if $\vz\in \aa_q$ with $q\in(1,\,\fz)$,
then $\vz\in \aa_r$ for any $r\in(q,\,\fz)$
and $\vz\in \aa_d$ for some $d\in(1,\,q)$.
Thus,
the {\it{critical weight index}} of $\vz\in\aa_\fz$ is defined as follows:
\begin{eqnarray}\label{e2.4}
q(\vz):=\inf\{q\in[1,\,\fz):\ \vz\in\aa_q\}.
\end{eqnarray}
Observe that, if $q(\vz)\in(1,\,\fz)$, then $\vz\notin\aa_{q(\vz)}$, and there exists $\vz\notin\aa_1$ such that $q(\vz)=1$ (see \cite{jn87}  for more details).

\begin{definition}\label{d2.3}{\rm\cite[Definition 2.1]{k14}}
A function $\vz: \rn\times[0,\,\fz) \to [0,\,\fz)$ is  called a {\it{growth function}}
if the following conditions are satisfied:
\begin{enumerate}
\item[\rm{(i)}] $\vz$ is a {Musielak-Orlicz function};
%
%
\item[\rm{(ii)}] $\vz\in\aa_\fz$;

\item[\rm{(iii)}] $\vz$ is of uniformly lower type $p$ for some $p\in(0,\,1]$ and of uniformly upper type $1$.
\end{enumerate}
\end{definition}

Throughout the paper, we always assume that $\varphi$ is a growth function.

Recall that the
\emph{Musielak-Orlicz space} $L^{\vz}$ is defined to be the space of all measurable functions $f$ such that,
for some $\eta\in(0,\,\fz)$,
$$\int_\rn \vz\lf(x,\, \frac{|f(x)|}{\eta}\r)\, dx<\fz$$ equipped with the (quasi-)norm
$$\|f\|_{L^\vz}:=\inf\lf\{ \eta\in(0,\,\fz):\ \int_\rn \vz\lf(x,\, \frac{|f(x)|}{\eta}\r)\, dx\le 1\r\}.$$

Similarly, the \emph{weak Musielak-Orlicz space} $WL^{\vz}$ is defined to be the space of all measurable functions $f$ such that, for some $\eta\in(0,\,\fz)$,
$$\sup_{t\in(0,\,\fz)} \vz\lf(\{|f|>t\},\, \frac{t}{\eta}\r)<\fz$$ equipped with the quasi-norm
$$ \|f\|_{WL^\vz}:=\inf\lf\{ \eta\in(0,\,\fz):\ \sup_{t\in(0,\,\fz)}\vz\lf(\{|f|>t\},\,\frac{t}{\eta}\r)\le 1\r\}. $$

In what follows, we denote by $\cs$ the {\it space of all Schwartz functions}
and by $\cs'$ its {\it dual space } (namely, the {\it space of all tempered
distributions}). For any $m\in\nn$, let $\cs_m$ be the {\it{space}} of all $\psi\in\cs$ such that $\|\psi\|_{\cs_m}\le1$, where
\begin{eqnarray*}
\|\psi\|_{\cs_m}:=\sup_{\az\in\nn^n,\,|\az|\le m+1}\sup_{x\in\rn}(1+|x|)^{(m+2)(n+1)}|\partial^\az\psi(x)|.
\end{eqnarray*}
Then, for any $m\in\nn$ and $f\in \cs'$, the {\it{non-tangential grand maximal function}}
$f^\ast_m$ of $f$ is defined by setting, for all $x\in\rn$,
\begin{eqnarray}\label{e2.m1}
f^\ast_m(x) := \sup_{\psi\in \cs_m}\,\sup_{|y-x|<t,\,t\in(0,\,\fz)}
|f\ast\psi_t(y)|,
\end{eqnarray}
where, for any $t\in(0,\,\fz)$, $\psi_t(\cdot):= t^{-n}\psi(\frac{\cdot}{t})$.
When
\begin{align}\label{e2.5}
m=m(\vz) :=\lf\lfloor n\lf(\frac{q(\vz)}{i(\vz)}-1\r)\r\rfloor,
\end{align}
we denote $f^\ast_m$ simply by $f^\ast$,
where $q(\vz)$ and $i(\vz)$ are as in \eqref{e2.4} and \eqref{e2.1}, respectively.

\begin{definition}\label{d2.5} \cite[Definition 2.2]{k14}
Let $\vz$ be a growth function as in Definition \ref{d2.3}.
The \emph{Musielak-Orlicz Hardy space} $H^\vz$ is defined as the space of all $f\in\cs'$
such that $f^\ast\in L^\vz$ endowed with the (quasi-)norm
$$\|f\|_{H^\vz}:=\|f^\ast\|_{L^\vz}.$$
\end{definition}

The main results of this paper are as follows, the proofs of which are given in Sections \ref{s3} and \ref{s4}.

\begin{theorem}\label{dingli.1}
Let $\Omega\in{\rm{Lip}}_\alpha(S^{n-1})$ with $\az\in(0,\,1]$, $\rho\in(n/2,\,\fz)$ and $\bz\in(0,\,\min\{1/2,\,\az,\,\rho-n/2\})$.
Suppose $\vz$ is a growth function as in Definition \ref{d2.3} with $p\in(n/(n+\bz),\,1]$.
If $\vz\in\aa_{p(1+{\bz}/{n})}$,
then there exists a positive constant $C$ independent of $f$ such that
$$\lf\|\mu^\rho_{\Omega,\,S}(f)\r\|_{L^\vz} \leq C\|f\|_{H^\vz}.$$
\end{theorem}

\begin{theorem}\label{dingli.2}
Let $\Omega\in{\rm{Lip}}_\alpha(S^{n-1})$ with $\az\in(0,\,1]$, $\rho\in(n/2,\,\fz)$ and $\bz\in(0,\,\min\{1/2,\,\az,\,\rho-n/2\})$.
Suppose $\vz$ is a growth function as in Definition \ref{d2.3} with $p:=n/(n+\bz)$ and $I(\vz)\in(0,\,1)$,
where $I(\vz)$ is as in \eqref{e2.1.1}. If $\vz\in\aa_1$,
then there exists a positive constant $C$ independent of $f$ such that
$$\lf\|\mu^\rho_{\Omega,\,S}(f)\r\|_{WL^\vz} \leq C\|f\|_{H^\vz}.$$
\end{theorem}

\begin{theorem}\label{dingli.3}
Let $\Omega\in{\rm{Lip}}_\alpha(S^{n-1})$ with $\az\in(0,\,1]$, $\rho\in(n/2,\,\fz)$, $\lz\in(2,\,\fz)$
and $\bz\in(0,\,\min\{1/2,\,\az,\,\rho-n/2,\,(\lz-2)n/3\})$.
Suppose $\vz$ is a growth function as in Definition \ref{d2.3} with $p\in(n/(n+\bz),\,1]$.
If $\vz\in\aa_{p(1+{\bz}/{n})}$,
then there exists a positive constant $C$ independent of $f$ such that
$$\lf\|\mu^{\rho,\,\ast}_{\Omega,\,\lz}(f)\r\|_{L^\vz} \leq C\|f\|_{H^\vz}.$$
\end{theorem}

\begin{theorem}\label{dingli.4}
Let $\Omega\in{\rm{Lip}}_\alpha(S^{n-1})$ with $\az\in(0,\,1]$, $\rho\in(n/2,\,\fz)$, $\lz\in(2,\,\fz)$
and $\bz\in(0,\,\min\{1/2,\,\az,\,\rho-n/2,\,(\lz-2)n/3\})$.
Suppose $\vz$ is a growth function as in Definition \ref{d2.3} with $p:=n/(n+\bz)$ and $I(\vz)\in(0,\,1)$,
where $I(\vz)$ is as in \eqref{e2.1.1}. If $\vz\in\aa_1$,
then there exists a positive constant $C$ independent of $f$ such that
$$\lf\|\mu^{\rho,\,\ast}_{\Omega,\,\lz}(f)\r\|_{WL^\vz} \leq C\|f\|_{H^\vz}.$$
\end{theorem}

\begin{remark}\label{r2.1}
\begin{enumerate}
%
%
\item[\rm{(i)}]
Let $\omega$ be a classic Muckenhoupt weight and $\phi$ an Orlicz function.

\quad(a) When $\vz(x,\,t):=\omega(x)\phi(t)$ for all $(x,\,t)\in\rn\times[0,\,\fz)$,
we have $H^\vz=H^\phi_\omega$. In this case,
Theorems \ref{dingli.1}-\ref{dingli.4} hold true for weighted Orlicz Hardy space. Even when $\varphi(x,\,t):=\phi(t)$,
the above results are also new.

\quad(b) When $\vz(x,\,t):=\omega(x)\,t^p$ for all $(x,\,t)\in\rn\times[0,\,\fz)$,
$H^\varphi$ is reduced to weighted Hardy space $H_\omega^p$. In this case,
Theorems \ref{dingli.1}-\ref{dingli.4} are new and part of these results even for Hardy space $H^p$ (namely, $\omega\equiv1$) are also new.

\item[\rm{(ii)}]
We only prove Theorems \ref{dingli.1} and \ref{dingli.4}, since the proofs of Theorems \ref{dingli.2} and \ref{dingli.3} are analogous.
\end{enumerate}
\end{remark}




\section{Proof of Theorem \ref{dingli.1}\label{s3}}

To show Theorem \ref{dingli.1}, we need some notions and auxiliary lemmas.

\begin{definition}\label{d2.11}{\rm\cite[Definition 2.4]{k14}}
Let $\vz$ be a growth function as in Definition \ref{d2.3}.
\begin{enumerate}
\item[\rm{(i)}] A triplet $(\vz,\,q,\,s)$ is said to be {\it {admissible}},
if $q\in (q(\vz),\,\fz]$ and $s \in [m(\vz),\,\fz)\cap\nn$,
where $q(\vz)$ and $m(\vz)$ are as in \eqref{e2.4} and \eqref{e2.5}, respectively.

\item[\rm{(ii)}] For an admissible triplet $(\vz,\,q,\,s)$, a measurable function $a$ is called a
{\it $(\vz,\,q,\,s)$-atom} if there exists some ball $B\subset\rn$ such that the following conditions are satisfied:

\quad(a) $a$ is supported in $B$;

\quad(b) $\|a\|_{L^q_\vz(B)}\leq\|\chi_B\|^{-1}_{L^\vz}$, where
\begin{eqnarray*}
\|a\|_{L_\vz^q(B)}:=
\lf\{\begin{array}{ll}
\dsup_{t\in(0,\,\fz)}\lf[\frac{1}{\vz(B,\,t)}\int_B|a(x)|^q \vz(x,\,t)\,dx\r]^{1/q}
                          ,\,\,\,&q\in[1,\,\fz),\\
\,\\
\|a\|_{L^\fz(B)},&q=\fz;
\end{array}\r.
\end{eqnarray*}

\quad(c) $\int_\rn a(x)x^\az dx=0$ for any $\az\in\nn^n$ with $|\az|\leq s$.

\end{enumerate}
\end{definition}

\begin{lemma}\label{m2}
Let $\Omega\in{\rm{Lip}}_\alpha(S^{n-1})$ with $\az\in(0,\,1]$, $\rho\in(n/2,\,\fz)$ and $\bz\in(0,\,\min\{1/2,\,\az,\,\rho-n/2\})$.
Suppose $b$ is a multiple of a $(\vz,\,\fz,\,s)$-atom associated with some ball $B:=B(x_0,\,r)$.
Then there exists a positive constant $C$ independent of $b$ such that, for any $x\in (64B)^\complement$,
\begin{align*}
\mu^\rho_{\Omega,\,S}(b)(x)\le C\|b\|_{L^\fz}\frac{r^{n+\beta}}{|x-x_0|^{n+\beta}}.
\end{align*}
\end{lemma}
\begin{proof}
We show this lemma by borrowing some ideas from the proof of \cite[Theorem 1]{dlx07a}.
The trick of the proof is to find a subtle segmentation.
For any $x\in (64B)^\complement$, write
\begin{align*}
\mu^\rho_{\Omega,\,S}(b)(x)
&=\lf(\int\int_{|y-x|<t}\lf|\int_{|y-z|<t}\frac{\Omega(y-z)}{|y-z|^{n-\rho}}b(z)\,dz\r|^2\,\frac{dydt}{t^{n+2\rho+1}}\r)^{\frac12} \\
&\le \lf(\int\int_{\substack{|y-x|<t \\ y\in 16B}}\lf|\int_{|y-z|<t}\frac{\Omega(y-z)}{|y-z|^{n-\rho}}b(z)\,dz\r|^2\,\frac{dydt}{t^{n+2\rho+1}}\r)^{\frac12} \\
&\hs+\lf(\int\int_{\substack{|y-x|<t \\ y\in (16B)^\complement \\ t\le|y-x_0|+8r}}\cdot\cdot\cdot\r)^{\frac12}
+\lf(\int\int_{\substack{|y-x|<t \\ y\in (16B)^\complement \\ t>|y-x_0|+8r}}\cdot\cdot\cdot\r)^{\frac12}
=:{\rm{I_1+I_2+I_3}}.
\end{align*}

For ${\rm{I_1}}$, by $x\in (64B)^\complement$, $y\in 16B$ and $z\in B$, we know that
$$t>|y-x|\ge|x-x_0|-|y-x_0|>|x-x_0|-\frac14|x-x_0|>\frac12|x-x_0| \ {\rm{and}} \ |y-z|<32r.$$
From this, $\Omega\in L^\fz(S^{n-1})$ (since $\Omega\in{\rm{Lip}}_\alpha(S^{n-1})$) and $\bz<\rho-n/2$, it follows that, for any $x\in (64B)^\complement$,
\begin{align*}
{\rm{I_1}}
&=\lf(\int\int_{\substack{|y-x|<t \\ y\in 16B}}\lf|\int_{|y-z|<t}\frac{\Omega(y-z)}{|y-z|^{n-\rho}}b(z)\,dz\r|^2\,\frac{dydt}{t^{n+2\rho+1}}\r)^{\frac12} \\
&\le \lf[\int\int_{\substack{y\in 16B \\ t>|x-x_0|/2}}\lf(\int_{|y-z|<32r}\frac{|\Omega(y-z)|}{|y-z|^{n-\rho}}|b(z)|\,dz\r)^2\,\frac{dydt}{t^{n+2\rho+1}}\r]^{\frac12} \\
&\le \|\Omega\|_{L^\fz(S^{n-1})}\|b\|_{L^\fz}\lf[\int\int_{\substack{y\in 16B \\ t>|x-x_0|/2}}\lf(\int_{|z|<32r}\frac{1}{|z|^{n-\rho}}\,dz\r)^2\,\frac{dydt}{t^{n+2\rho+1}}\r]^{\frac12} \\
&\ls \|b\|_{L^\fz}\lf(\int_{16B}1\,dy\,\int^\fz_{\frac{|x-x_0|}{2}}\frac{dt}{t^{n+2\rho+1}}\r)^{\frac12}\lf(\int_{S^{n-1}}\int^{32r}_0\frac{1}{u^{n-\rho}}u^{n-1}\,dud\sigma(z')\r) \\
&\sim \|b\|_{L^\fz}\frac{r^{\rho+n/2}}{|x-x_0|^{\rho+n/2}}\ls \|b\|_{L^\fz}\frac{r^{n+\bz}}{|x-x_0|^{n+\bz}},
\end{align*}
which is wished.

For ${\rm{I_2}}$, by $x\in (64B)^\complement$, $y\in (16B)^\complement$, $z\in B$ and the mean value theorem, we know that
\begin{align}\label{q1}
|y-z|\sim|y-x_0|;
\end{align}
\begin{align}\label{q2}
|y-x_0|-2r\le|y-x_0|-|x_0-z|\le|y-z|<t\le|y-x_0|+8r;
\end{align}
\begin{align}\label{q3}
|x-x_0|\le|x-y|+|y-x_0|\le t+|y-x_0|\le 2|y-x_0|+8r\le 3|y-x_0|;
\end{align}
\begin{align}\label{q4}
\lf|\frac{1}{(|y-x_0|-2r)^{n+2\rho}}-\frac{1}{(|y-x_0|+8r)^{n+2\rho}}\r|\ls\frac{r}{|y-x_0|^{n+2\rho+1}}.
\end{align}
From Minkowski's inequality for integrals, \eqref{q1}-\eqref{q4}, $\Omega\in L^\fz(S^{n-1})$ and $\bz<1/2$, we deduce that, for any $x\in (64B)^\complement$,
\begin{align*}
{\rm{I_2}}
&=\lf(\int\int_{\substack{|y-x|<t \\ y\in (16B)^\complement \\ t\le|y-x_0|+8r}}\lf|\int_{|y-z|<t}\frac{\Omega(y-z)}{|y-z|^{n-\rho}}b(z)\,dz\r|^2\,\frac{dydt}{t^{n+2\rho+1}}\r)^{\frac12} \\
&\le \int_B|b(z)| \lf(\int\int_{\substack{|y-x|<t \\ y\in (16B)^\complement \\ t\le|y-x_0|+8r \\ |y-z|<t}}\frac{|\Omega(y-z)|^2}{|y-z|^{2n-2\rho}}\,
\frac{dydt}{t^{n+2\rho+1}}\r)^{\frac12}\,dz \\
&\ls \int_B|b(z)| \lf[\int_{\substack{y\in (16B)^\complement \\ |x-x_0|\le3|y-x_0|}}\frac{|\Omega(y-z)|^2}{|y-x_0|^{2n-2\rho}}
\lf(\int^{|y-x_0|+8r}_{|y-x_0|-2r}\frac{dt}{t^{n+2\rho+1}}\r)\,dy\r]^{\frac12}\,dz \\
&\ls \int_B|b(z)| \lf(\int_{\substack{y\in (16B)^\complement \\ |x-x_0|\le3|y-x_0|}}\frac{|\Omega(y-z)|^2}{|y-x_0|^{2n-2\rho}}
\frac{r}{|y-x_0|^{n+2\rho+1}}\,dy\r)^{\frac12}\,dz \\
&\ls \int_B|b(z)| \lf(\int_{\substack{y\in (16B)^\complement \\ |x-x_0|\le3|y-x_0|}}\frac{1}{|y-x_0|^{n-2\bz+1}}
\frac{r}{|x-x_0|^{2n+2\bz}}\,dy\r)^{\frac12}\,dz \\
&\ls \|b\|_{L^\fz}\frac{r^{n+1/2}}{|x-x_0|^{n+\bz}}\lf(\int_{S^{n-1}}\int^\fz_r\frac{1}{u^{n-2\bz+1}}u^{n-1}\,dud\sigma(y')\r)^\frac12 \\
&\sim \|b\|_{L^\fz}\frac{r^{n+1/2}}{|x-x_0|^{n+\bz}}\,r^{\bz-1/2} \sim \|b\|_{L^\fz}\frac{r^{n+\bz}}{|x-x_0|^{n+\bz}},
\end{align*}
which is also wished.

For ${\rm{I_3}}$, noticing that $t>|y-x_0|+8r$, we see that, for any $y\in (16B)^\complement$,
\begin{align}\label{q5}
B\subset \{z\in\rn: \ |z-y|<t\}.
\end{align}
On the other hand, we claim that, for any $y\in (16B)^\complement$ and $z\in B$,
\begin{align}\label{q6}
\lf|\frac{\Omega(y-z)}{|y-z|^{n-\rho}}-\frac{\Omega(y-x_0)}{|y-x_0|^{n-\rho}}\r|\ls \frac{|z-x_0|^\az}{|y-x_0|^{n-\rho+\az}}.
\end{align}
Indeed, by the mean value theorem and $\Omega\in{\rm{Lip}}_\alpha(S^{n-1})$ with $\az\in(0,\,1]$, we obtain that, for any $y\in (16B)^\complement$ and $z\in B$,
\begin{align*}
\lf|\frac{\Omega(y-z)}{|y-z|^{n-\rho}}-\frac{\Omega(y-x_0)}{|y-x_0|^{n-\rho}}\r|
&\le\lf|\frac{\Omega(y-z)}{|y-z|^{n-\rho}}-\frac{\Omega(y-z)}{|y-x_0|^{n-\rho}}\r|+\lf|\frac{\Omega(y-z)}{|y-x_0|^{n-\rho}}-\frac{\Omega(y-x_0)}{|y-x_0|^{n-\rho}}\r| \\
&\ls\lf|\frac{1}{|y-z|^{n-\rho}}-\frac{1}{|y-x_0|^{n-\rho}}\r| \\
&\hs+\frac{1}{|y-x_0|^{n-\rho}} \lf|\Omega\lf(\frac{y-z}{|y-z|}\r)-\Omega\lf(\frac{y-x_0}{|y-x_0|}\r)\r| \\
&\ls\frac{|z-x_0|}{|y-x_0|^{n-\rho+1}}+\frac{1}{|y-x_0|^{n-\rho}} \lf|\frac{y-z}{|y-z|}-\frac{y-x_0}{|y-x_0|}\r|^\az \\
&\ls\frac{1}{|y-x_0|^{n-\rho}}\frac{|z-x_0|}{|y-x_0|}+\frac{1}{|y-x_0|^{n-\rho}}\lf(\frac{|z-x_0|}{|y-x_0|}\r)^\az \\
&\ls\frac{|z-x_0|^\az}{|y-x_0|^{n-\rho+\az}}.
\end{align*}
From \eqref{q5}, the vanishing moments of $b$, Minkowski's inequality for integrals and \eqref{q6}, it follows that, for any $x\in (64B)^\complement$,
\begin{align*}
{\rm{I_3}}
&=\lf[\int\int_{\substack{|y-x|<t \\ y\in (16B)^\complement \\ t>|y-x_0|+8r}}\lf|\int_{|y-z|<t}
\lf(\frac{\Omega(y-z)}{|y-z|^{n-\rho}}-\frac{\Omega(y-x_0)}{|y-x_0|^{n-\rho}}\r)b(z)\,dz\r|^2\frac{dydt}{t^{n+2\rho+1}}\r]^{\frac12} \\
&\le\int_B |b(z)|
\lf(\int\int_{\substack{y\in (16B)^\complement \\ t>\max\{|y-x|,\,|y-x_0|+8r,\,|y-z|\}}}
\lf|\frac{\Omega(y-z)}{|y-z|^{n-\rho}}-\frac{\Omega(y-x_0)}{|y-x_0|^{n-\rho}}\r|^2\frac{dydt}{t^{n+2\rho+1}}\r)^\frac12dz \\
&\le C\int_B |b(z)|
\lf(\int\int_{\substack{y\in (16B)^\complement \\ t>\max\{|y-x|,\,|y-x_0|+8r,\,|y-z|\}}}
\frac{|z-x_0|^{2\az}}{|y-x_0|^{2n-2\rho+2\az}}\frac{dydt}{t^{n+2\rho+1}}\r)^\frac12dz \\
&\le C\int_B |b(z)|
\lf(\int\int_{\substack{y\in (16B)^\complement \\ t>\max\{|y-x|,\,|y-x_0|+8r,\,|y-z|\} \\ |x-x_0|\le2|y-x_0|}}
\frac{|z-x_0|^{2\az}}{|y-x_0|^{2n-2\rho+2\az}}\frac{dydt}{t^{n+2\rho+1}}\r)^\frac12dz \\
&\hs+C\int_B |b(z)|
\lf(\int\int_{\substack{y\in (16B)^\complement \\ t>\max\{|y-x|,\,|y-x_0|+8r,\,|y-z|\} \\ |x-x_0|>2|y-x_0|}}
\cdot\cdot\cdot\r)^\frac12dz=:C({\rm{I_{31}}+\rm{I_{32}}}).
\end{align*}

Below, we will give the estimates of ${\rm{I_{31}}}$ and ${\rm{I_{32}}}$, respectively.

For ${\rm{I_{31}}}$, by $\bz<\az$, we know that, for any $x\in (64B)^\complement$,
\begin{align*}
{\rm{I_{31}}}
&\le\int_B |b(z)|
\lf(\int\int_{\substack{y\in (16B)^\complement \\ t>|y-x_0| \\ |x-x_0|\le2|y-x_0|}}
\frac{|z-x_0|^{2\az}}{|y-x_0|^{2n-2\rho+2\az}}\frac{dydt}{t^{n+2\rho+1}}\r)^\frac12dz \\
&\le \int_B |b(z)|
\lf[\int_{\substack{y\in (16B)^\complement \\ |x-x_0|\le2|y-x_0|}}\frac{r^{2\az}}{|y-x_0|^{2n-2\rho+2\az}}
\lf(\int^\fz_{|y-x_0|}\frac{dt}{t^{n+2\rho+1}}\r)dy\r]^\frac12dz \\
&\le \|b\|_{L^\fz}r^{n+\az} \lf(\int_{\substack{y\in (16B)^\complement \\ |x-x_0|\le2|y-x_0|}}\frac{1}{|y-x_0|^{3n+2\az}}dy\r)^\frac12 \\
&\ls \|b\|_{L^\fz}r^{n+\az} \lf(\int_{\substack{y\in (16B)^\complement \\ |x-x_0|\le2|y-x_0|}}
\frac{1}{|y-x_0|^{n-2\bz+2\az}}\frac{1}{|x-x_0|^{2n+2\bz}}dy\r)^\frac12 \\
&\ls \|b\|_{L^\fz}\frac{r^{n+\az}}{|x-x_0|^{n+\bz}} \lf(\int_{B^\complement}\frac{1}{|y-x_0|^{n-2\bz+2\az}}dy\r)^\frac12 \\
&\sim \|b\|_{L^\fz}\frac{r^{n+\az}}{|x-x_0|^{n+\bz}}\lf(\int_{S^{n-1}}\int^\fz_r\frac{1}{u^{n-2\bz+2\az}}u^{n-1}\,dud\sigma(y')\r)^\frac12 \\
&\sim \|b\|_{L^\fz}\frac{r^{n+\az}}{|x-x_0|^{n+\bz}}\,r^{\bz-\az}\sim \|b\|_{L^\fz}\frac{r^{n+\bz}}{|x-x_0|^{n+\bz}}.
\end{align*}

For ${\rm{I_{32}}}$, noticing that $t>\max\{|y-x|,\,|y-x_0|+8r,\,|y-z|\}$ and $|x-x_0|>2|y-x_0|$, we see that
$$t>|y-x|\ge|x-x_0|-|y-x_0|>\frac12|x-x_0|.$$
From this and $\bz<\min\{\az,\,\rho-n/2\}$, it follows that, for any $x\in (64B)^\complement$,
\begin{align*}
{\rm{I_{32}}}
&=\int_B |b(z)|
\lf(\int\int_{\substack{y\in (16B)^\complement \\ t>\max\{|y-x|,\,|y-x_0|+8r,\,|y-z|\} \\ |x-x_0|>2|y-x_0|}}
\frac{|z-x_0|^{2\az}}{|y-x_0|^{2n-2\rho+2\az}}\frac{dydt}{t^{n+2\rho+1}}\r)^\frac12dz \\
&\le\int_B |b(z)|
\lf(\int\int_{\substack{y\in (16B)^\complement \\ t>|y-x_0|+8r  \\ t>|x-x_0|/2}}
\frac{r^{2\az}}{|y-x_0|^{2n-2\rho+2\az}}\frac{dydt}{t^{n+2\rho+1}}\r)^\frac12dz \\
&\le \|b\|_{L^\fz}r^{n+\az}\lf[\int_{(16B)^\complement}\frac{1}{|y-x_0|^{2n-2\rho+2\az}}
\lf(\int_{\substack{ t>|y-x_0|  \\ t>|x-x_0|/2}}\frac{1}{t^{n+2\rho+1}}dt\r)dy\r]^\frac12 \\
&\ls \|b\|_{L^\fz}r^{n+\az}\lf[\int_{(16B)^\complement}\frac{1}{|y-x_0|^{2n-2\rho+2\az}}
\lf(\int_{\substack{ t>|y-x_0|  \\ t>|x-x_0|/2}}\frac{t^{2n+2\bz}}{|x-x_0|^{2n+2\bz}}\frac{1}{t^{n+2\rho+1}}dt\r)dy\r]^\frac12 \\
&\ls \|b\|_{L^\fz}\frac{r^{n+\az}}{|x-x_0|^{n+\bz}}\lf[\int_{(16B)^\complement}\frac{1}{|y-x_0|^{2n-2\rho+2\az}}
\lf(\int_{|y-x_0|}^\fz\frac{1}{t^{-n+2\rho-2\bz+1}}dt\r)dy\r]^\frac12 \\
&\ls \|b\|_{L^\fz}\frac{r^{n+\az}}{|x-x_0|^{n+\bz}}\lf(\int_{B^\complement}\frac{1}{|y-x_0|^{n+2\az-2\bz}}dy\r)^\frac12 \\
&\sim \|b\|_{L^\fz}\frac{r^{n+\az}}{|x-x_0|^{n+\bz}}\lf(\int_{S^{n-1}}\int^\fz_r\frac{1}{u^{n+2\az-2\bz}}u^{n-1}\,dud\sigma(y')\r)^\frac12 \\
&\sim \|b\|_{L^\fz}\frac{r^{n+\az}}{|x-x_0|^{n+\bz}}\,r^{\bz-\az}\sim \|b\|_{L^\fz}\frac{r^{n+\bz}}{|x-x_0|^{n+\bz}}.
\end{align*}

Combining the estimates of ${\rm{I_1}}$, ${\rm{I_2}}$, ${\rm{I_{31}}}$ and ${\rm{I_{32}}}$, we obtain the desired inequality.
This finishes the proof of Lemma \ref{m2}.
\end{proof}

%


\begin{proof}[Proof of Theorem \ref{dingli.1}]
Obviously, $\mu^\rho_{\Omega,\,S}$ is a positive sublinear operator and bounded on $L^2$.
Thus, by the boundedness criterions of operators from $H^\vz$ to $L^\vz$ (see \cite[Lemma 3.12]{lll17}),
Theorem \ref{dingli.1} will be proved by showing that $\mu^\rho_{\Oz,\,S}$ maps all multiple of a $(\vz,\,\fz,\,s)$-atoms
into uniformly bounded elements of $L^\vz$, namely,
there exists a positive constant $C$ such that,
for any $\eta\in(0,\,\fz)$ and multiple of a $(\vz,\,\fz,\,s)$-atom $b$
associated with some ball $B(x_0,\,r)\subset\rn$,
$$\int_{\rn}\vz\lf(x,\,\frac{\mu^\rho_{\Omega,\,S}(b)(x)}{\eta}\r)\,dx \le C\vz\lf(B,\,\frac{\|b\|_{L^\fz}}{\eta}\r).$$

For any $\eta\in(0,\,\fz)$, write
\begin{align*}
\int_{\rn}\vz\lf(x,\,\frac{\mu^\rho_{\Omega,\,S}(b)(x)}{\eta}\r)\,dx
=\int_{64B}\vz\lf(x,\,\frac{\mu^\rho_{\Omega,\,S}(b)(x)}{\eta}\r)\,dx+\int_{(64B)^\complement}\cdot\cdot\cdot=:{\rm{I_1+I_2}}.
\end{align*}

For ${\rm{I_1}}$, noticing that $p>n/(n+\bz)$, we see that $\vz\in\aa_2$.
From the uniformly upper type 1 property of $\vz$,
the boundedness on $L^2_{\vz(\cdot,\,t)}$, uniformly in $t\in(0,\,\fz)$, of $\mu^\rho_{\Omega,\,S}$ with $\vz\in\aa_2$
(see \cite[Theorem 1 and Remark 2]{xd07}), and Lemma \ref{bcd}(i) with $\vz\in\aa_2$,
it follows that, for any $\eta\in(0,\,\fz)$,
\begin{align*}
{\mathrm{I_1}}
&\ls \int_{64B}\lf(1+\frac{|\mu^\rho_{\Omega,\,S}(b)(x)|}{\|b\|_{L^\fz}}\r)^2\vz\lf(x,\,\frac{\|b\|_{L^\fz}}{\eta}\r)\,dx \\
&\ls \int_{64B}\lf(1+\frac{|\mu^\rho_{\Omega,\,S}(b)(x)|^2}{\|b\|^2_{L^\fz}}\r)\vz\lf(x,\,\frac{\|b\|_{L^\fz}}{\eta}\r)\,dx \\
&\ls \vz\lf(64B,\,\frac{\|b\|_{L^\fz}}{\eta}\r)+\frac{1}{\|b\|^2_{L^\fz}}\int_\rn |\mu^\rho_{\Omega,\,S}(b)(x)|^2 \vz\lf(x,\,\frac{\|b\|_{L^\fz}}{\eta}\r)\,dx \\
&\ls \vz\lf(64B,\,\frac{\|b\|_{L^\fz}}{\eta}\r)+\frac{1}{\|b\|^2_{L^\fz}}\int_{B} |b(x)|^2 \vz\lf(x,\,\frac{\|b\|_{L^\fz}}{\eta}\r)\,dx \\
&\ls \vz\lf(B,\,\frac{\|b\|_{L^\fz}}{\eta}\r).
\end{align*}

For ${\rm{I_2}}$, by Lemma \ref{m2}, the uniformly lower type $p$ property of $\vz$, and Lemma \ref{bcd}(ii) with $\vz\in\aa_{p(1+{\bz}/{n})}$,
we know that, for any $\eta\in(0,\,\fz)$,
\begin{align*}
{\mathrm{I_2}}
&\ls \int_{(64B)^\complement}\vz\lf(x,\,\frac{\|b\|_{L^\fz}}{\eta}\frac{r^{n+\bz}}{|x-x_0|^{n+\bz}}\r)\,dx \\
&\ls r^{(n+\bz)p}\int_{B^\complement}\frac{1}{|x-x_0|^{(n+\bz)p}}\vz\lf(x,\,\frac{\|b\|_{L^\fz}}{\eta}\r)\,dx\ls \vz\lf(B,\,\frac{\|b\|_{L^\fz}}{\eta}\r).
\end{align*}

To summarize what we have proved, we obtain the desired inequality.
This finishes the proof of Theorem \ref{dingli.1}.
\end{proof}

\section{Proof of Theorem \ref{dingli.4}\label{s4}}
To show Theorem \ref{dingli.4}, we first need to establish
a boundedness criterion of operators from $H^\vz$ to $WL^\vz$ (see Lemma \ref{yt2} below).

The following lemma comes from \cite[Lemma 7.13]{bckyy13}.
When $\vz(x,\,t):=t^p$ for all $(x,\,t)\in\rn\times[0,\,\fz)$ with $p\in(0,\,1)$,
it is reduced to the well-known superposition principle of weak type estimates obtained by Stein et al.,
\cite[Lemma (1.8)]{stw81} and, independently, by Kalton \cite[Theorem 6.1]{k80}.

\begin{lemma}\label{dj}
Let $\vz$ be a growth function as in Definition \ref{d2.3}
satisfying $I(\vz)\in(0,\,1)$, where $I(\vz)$ is as in \eqref{e2.1.1}.
Assume that $\{f_j\}_{j\in\zz+}$ is a sequence of measurable functions such that, for some $\eta\in(0,\,\fz)$,
$$\sum_{j\in\zz+}\sup_{\az\in(0,\,\fz)}\vz\lf(\{|f_j|>\az\},\,\frac{\az}{\eta}\r)<\fz.$$
Then there exists a positive constant $C$, depending only on $\vz$, such that, for any $\bz\in(0,\,\fz)$,
$$\vz\lf(\lf\{ \sum_{j\in\zz+}|f_j|>\bz \r\},\,\frac{\bz}{\eta}\r)
\le C\sum_{j\in\zz+}\sup_{\az\in(0,\,\fz)}\vz\lf(\{|f_j|>\az\},\,\frac{\az}{\eta}\r).$$
\end{lemma}

Using the same argument as in the proof of \cite[Lemma 4.3]{k14},
we can easily carry out the proof of the following lemma, the details being omitted.
\begin{lemma}\label{fs1}
Let $\vz$ be a growth function as in Definition \ref{d2.3}.
For a given positive constant $\tilde{C}$,
there exists a positive constant $C$ such that, for any $\eta\in(0,\,\fz)$,
$$\sup_{\alpha\in(0,\,\fz)}\vz\lf(\{|f|>\alpha\},\,\frac{\alpha}{\eta}\r)\le\tilde{C} \ implies \ that \ \|f\|_{WL^\vz}\le C \eta.$$
\end{lemma}

\begin{definition}\label{ddd} \cite[p.122]{k14}
For an admissible triplet $(\vz,\,q,\,s)$,
the {\emph{Musielak-Orlicz atomic Hardy space}} $H^{\vz,\,q,\,s}_{\rm{at}}$ is defined as the space of all
$f \in \cs'$ which can be represented as a linear combination of $(\vz,\, q,\, s)$-atoms, that is,
$f =\sum_j b_j$ in $\cs'$, where $b_j$ for each $j$ is a multiple of some $(\vz,\, q,\, s)$-atom
supported in some ball ${B_j}$, with the property
\begin{align*}
\sum_{j}\vz\lf({B_j},\, \|b_j\|_{L^q_\vz({B_j})}\r)<\fz.
\end{align*}
For any given sequence of multiples of $(\vz,\,q,\,s)$-atoms, $\{b_j\}_j$, let
$$ \Lz_q(\{b_j\}_j):=\inf\lf\{\eta\in(0,\,\fz):\ \sum_j \vz\lf({B_j},\,\frac{\|b_j\|_{L^q_\vz({B_j})}}{\eta}\r)\le 1 \r\}  $$
and then the (quasi-)norm of $f\in\cs'$ is defined by
$$\|f\|_{H^{\vz,\,q,\,s}_{\rm{at}}}:=\inf\lf\{\Lz_q\lf(\{b_j\}_j\r)\r\},  $$
where the infimum is taken over all admissible decompositions of $f$ as above.
\end{definition}

\begin{lemma}\label{yzfj}{\rm{\cite[Theorem 3.1]{k14}}}
Let $(\vz,\,q,\,s)$ be an admissible triplet as in Definition \ref{d2.11}. Then
$$H^\vz=H^{\vz,\,q,\,s}_{{\rm{at}}}$$
with equivalent (quasi-)norms.
\end{lemma}

Recall that a {\it quasi-Banach space} $\mathcal{B}$ is a linear space endowed with a quasi-norm
$\|\cdot\|_\mathcal{B}$ which is nonnegative, non-degenerate (i.e., $\|f\|_\mathcal{B}=0$
if and only if $f={\bf 0}$), homogeneous, and obeys the quasi-triangle inequality, i.e., there
exists a constant $K$ no less than 1 such that, for any $f, g\in\mathcal{B}$,
$\|f+g\|_\mathcal{B}\leq K\lf(\|f\|_\mathcal{B}+\|g\|_\mathcal{B}\r)$.

\begin{lemma}\label{ar}
Let $\mathcal{B}$ be a quasi-Banach space equipped with the quasi-norm $\|\cdot\|_\mathcal{B}$.
For any $\{f_k\}_{k\in\zz_+}\subset\mathcal{B}$ and $f\in\mathcal{B}$,
if $\mathop{\lim}\limits_{k\rightarrow\fz} \lf\|f_k-f\r\|_\mathcal{B}=0$, then
$$\frac{1}{K}\lf\|f\r\|_\mathcal{B}\le \liminf_{k\rightarrow\fz}\lf\|f_k\r\|_\mathcal{B}\le
\limsup_{k\rightarrow\fz}\lf\|f_k\r\|_\mathcal{B}\le K\lf\|f\r\|_\mathcal{B},$$
where $K$ is a constant associated with $\cb$ as above.
\end{lemma}

\begin{proof}
The proof of this lemma is not particularly difficult and so is omitted.
\end{proof}

\begin{remark}\label{rr3.1}
In the process of the proof of this lemma, if we use Aoki-Rolewicz theorem (see \cite{s84}),
the result of Lemma \ref{ar} would be better.
Precisely, under the same assumptions of Lemma \ref{ar}, we can get
$\lim_{k\rightarrow\fz}\lf\|f_k\r\|_\mathcal{B}=\lf\|f\r\|_\mathcal{B}.$
However, Lemma \ref{ar} is just enough for later use.
\end{remark}

The following lemma gives a boundedness criterion of operators from $H^\vz$ to $WL^\vz$.
\begin{lemma}\label{yt2}
Let $\vz$ be a growth function as in Definition \ref{d2.3} satisfying
$I(\vz)\in(0,\,1)$, where $I(\vz)$ is as in \eqref{e2.1.1}.
Suppose that a linear or a positive sublinear operator $T$ is bounded on $L^2$.
If there exists a positive constant $C$ such that,
for any $\eta\in(0,\,\fz)$ and multiple of a $(\vz,\,q,\,s)$-atom $b$ associated with some ball $B\subset\rn$,
\begin{align}\label{z11}
 \sup_{\alpha\in(0,\,\fz)}\vz\lf(\lf\{|T(b)|>\alpha\r\},\,\frac{\alpha}{\eta}\r)
 \le C\vz\lf(B,\,\frac{\|b\|_{L^q_\vz(B)}}{\eta}\r),
\end{align}
then $T$ extends uniquely to a bounded operator from ${H^\vz}$ to $WL^\vz$.
\end{lemma}
\begin{proof}
We first assume that $f\in H^\vz\cap L^2$.
By the well known Calder\'on reproducing formula (see also \cite[Theorem 2.14]{lfy15}),
we know that there exists a sequence of multiples of $(\vz,\,q,\,s)$-atoms $\{b_j\}_{j\in\zz_+}$
associated with balls $\{B_{j}\}_{j\in\zz_+}$ such that
\begin{align}\label{a2}
f=\lim_{k\rightarrow\fz}\sum_{j=1}^k b_j
=:\lim_{k\rightarrow\fz}f_k \ {\rm{in}} \ \cs' \ {\rm{and \ also \ in}} \ L^2.
\end{align}
From the assumption that the linear or positive sublinear operator $T$
is bounded on $L^2$ and \eqref{a2}, it follows that
\begin{align*}
\lim_{k\rightarrow\fz}\lf\|T(f)-T\lf(f_k\r)\r\|_{L^2}
\le \lim_{k\rightarrow\fz}\lf\|T\lf(f-f_k\r)\r\|_{L^2}\ls \lim_{k\rightarrow\fz}\lf\|f-f_k\r\|_{L^2}\sim0,
\end{align*}
which, together with Riesz's theorem, implies that
\begin{align*}
T(f)=\lim_{k\rightarrow\fz}T\lf(f_k\r)
\le\lim_{k\rightarrow\fz}\sum_{j=1}^kT\lf(b_j\r)
=\sum_{j=1}^\fz T\lf(b_j\r) \ {\rm{almost \ everywhere}}.
\end{align*}
By this, Lemma \ref{dj} and \eqref{z11}, we obtain that, for any $\alpha\in(0,\,\fz)$,
\begin{align*}
\vz\lf(\{|T(f)|>\alpha\},\,\frac{\alpha}{\Lz_q(\{b_j\}_j)}\r)
&\le \vz\lf(\lf\{\sum_{j=1}^\fz |T\lf(b_j\r)|>\alpha\r\},\,\frac{\alpha}{\Lz_q(\{b_j\}_j)}\r) \\
&\ls \sum_{j=1}^\fz\sup_{\alpha\in(0,\,\fz)}\vz\lf(\lf\{|T\lf(b_j\r)|>\alpha\r\},\,\frac{\alpha}{\Lz_q(\{b_j\}_j)}\r)  \\
&\ls \sum_{j=1}^{\fz}\vz\lf(B_j,\,
\frac{\|b_j\|_{L^q_\vz(B_{j})}}{\Lz_q(\{b_j\}_j)}\r)\ls1,
\end{align*}
which, together with Lemma \ref{fs1}, further implies that
$$\|T(f)\|_{WL^\vz}\ls \Lz_q(\{b_j\}_j).$$
Taking infimum for all admissible decompositions of $f$ as above and using Lemma \ref{yzfj}, we obtain that, for any $f\in H^\vz\cap L^2$,
\begin{align}\label{z13}
\|T(f)\|_{WL^\vz} \ls \|f\|_{H^{\vz,\,q,\,s}_{{\rm{at}}}} \sim \|f\|_{H^\vz}.
\end{align}

Next, suppose $f\in{H^\vz}$. By the fact that $H^\vz\cap L^2$ is dense in $H^\vz$ (see {\cite[Remark 4.1.4]{ylk17}}),
we know that there exists a sequence of $\{f_j\}_{j\in\zz_+}\subset H^\vz\cap L^2$ such that
$f_j\rightarrow f$ as $j\rightarrow\fz$ in ${H^\vz}$.
Therefore, $\{f_j\}_{j\in\zz_+}$ is a Cauchy sequence in $H^\vz$.
From this and \eqref{z13}, we deduce that, for any $j,\,k\in\zz_+$,
$$\lf\|T(f_j)-T(f_k)\r\|_{{WL}^\vz} \leq \lf\|T(f_j-f_k)\r\|_{{WL}^\vz}\ls \lf\|f_j-f_k\r\|_{H^\vz}.$$
By this, we know that
$\{T(f_j)\}_{j\in\zz_+}$ is also a Cauchy sequence in ${WL}^\vz$ and hence
there exists some $g\in {WL}^\vz$ such that $T(f_j)\rightarrow g$ as $j\rightarrow\fz$ in ${WL}^\vz$.
Consequently, define $T(f):=g$. Below, we claim that $T(f)$ is well defined. Indeed,
for any other sequence $\{f'_j\}_{j\in\zz_+}\subset H^\vz\cap L^2$ satisfying
$f'_j\rightarrow f$ as $j\rightarrow\fz$ in ${H^\vz}$, by \eqref{z13},
we have
\begin{align*}
\lf\|T(f'_j)-T(f)\r\|_{WL^\vz}
&\ls \lf\|T(f'_j)-T(f_j)\r\|_{WL^\vz}+\lf\|T(f_j)-g\r\|_{WL^\vz} \\
&\ls \lf\|T(f'_j-f_j)\r\|_{WL^\vz}+\lf\|T(f_j)-g\r\|_{WL^\vz} \\
&\ls \lf\|f'_j-f_j\r\|_{H^\vz}+\lf\|T(f_j)-g\r\|_{WL^\vz} \\
&\ls \lf\|f'_j-f\r\|_{H^\vz}+\lf\|f-f_j\r\|_{H^\vz}+\lf\|T(f_j)-g\r\|_{WL^\vz}\rightarrow0 \ {\rm{as}} \ j\rightarrow\fz,
\end{align*}
which is wished. From this, Lemma \ref{ar} and \eqref{z13}, it follows that
$$\|T(f)\|_{{WL}^\vz}=\|g\|_{{WL}^\vz}\ls\liminf_{j\rightarrow\fz}\|T(f_j)\|_{{WL}^\vz}
\ls \limsup_{j\rightarrow\fz}\|f_j\|_{H^\vz}\ls\|f\|_{H^\vz}.$$
This finishes the proof of Lemma \ref{yt2}.
\end{proof}

\begin{lemma}\label{m4}
Let $\Omega\in{\rm{Lip}}_\alpha(S^{n-1})$ with $\az\in(0,\,1]$, $\rho\in(n/2,\,\fz)$, $\lz\in(2,\,\fz)$
and $\bz\in(0,\,\min\{1/2,\,\az,\,\rho-n/2,\,(\lz-2)n/3\})$.
Suppose $b$ is a multiple of a $(\vz,\,\fz,\,s)$-atom associated with some ball $B:=B(x_0,\,r)$.
Then there exists a positive constant $C$ independent of $b$ such that, for any $x\in (64B)^\complement$,
\begin{align*}
\mu^{\rho,\,\ast}_{\Omega,\,\lz}(b)(x)\le C\|b\|_{L^\fz}\frac{r^{n+\beta}}{|x-x_0|^{n+\beta}}.
\end{align*}
\end{lemma}

\begin{proof}
We show this lemma by borrowing some ideas from the proof of \cite[Theorem 1.1]{dlx07lp}.
By Lemma \ref{m2}, we know that, for any $x\in (64B)^\complement$,
\begin{align*}
\mu^{\rho,\,\ast}_{\Omega,\,\lz}(b)(x)
&=\lf[\int\int_{{\mathbb{R}}^{n+1}_+}\lf(\frac{t}{t+|x-y|}\r)^{\lz n}\lf|\int_{|y-z|<t}\frac{\Omega(y-z)}{|y-z|^{n-\rho}}b(z)\,dz\r|^2\,\frac{dydt}{t^{n+2\rho+1}}\r]^{\frac12} \\
&\le\lf[\int\int_{|y-x|<t}\lf(\frac{t}{t+|x-y|}\r)^{\lz n}\lf|\int_{|y-z|<t}\frac{\Omega(y-z)}{|y-z|^{n-\rho}}b(z)\,dz\r|^2\,\frac{dydt}{t^{n+2\rho+1}}\r]^{\frac12} \\
&\hs+\lf[\int\int_{|y-x|\ge t}\lf(\frac{t}{t+|x-y|}\r)^{\lz n}\lf|\int_{|y-z|<t}\frac{\Omega(y-z)}{|y-z|^{n-\rho}}b(z)\,dz\r|^2\,\frac{dydt}{t^{n+2\rho+1}}\r]^{\frac12} \\
&\le\mu^\rho_{\Omega,\,S}(b)(x) \\
&\hs+\lf[\int\int_{|y-x|\ge t}\lf(\frac{t}{t+|x-y|}\r)^{\lz n}\lf|\int_{|y-z|<t}\frac{\Omega(y-z)}{|y-z|^{n-\rho}}b(z)\,dz\r|^2\,\frac{dydt}{t^{n+2\rho+1}}\r]^{\frac12} \\
&\le C\|b\|_{L^\fz}\frac{r^{n+\beta}}{|x-x_0|^{n+\beta}} \\
&\hs+\lf[\int\int_{|y-x|\ge t}\lf(\frac{t}{t+|x-y|}\r)^{\lz n}\lf|\int_{|y-z|<t}\frac{\Omega(y-z)}{|y-z|^{n-\rho}}b(z)\,dz\r|^2\,\frac{dydt}{t^{n+2\rho+1}}\r]^{\frac12} \\
&=:C\|b\|_{L^\fz}\frac{r^{n+\beta}}{|x-x_0|^{n+\beta}}+{\rm{J}}.
\end{align*}
Thus, to show Lemma \ref{m4}, it suffices to prove that, for any $x\in (64B)^\complement$,
$${\rm{J}}\ls \|b\|_{L^\fz}\frac{r^{n+\beta}}{|x-x_0|^{n+\beta}}.$$

For any $x\in (64B)^\complement$, write
\begin{align*}
{\rm{J}}
&\le \lf[\int\int_{\substack{|y-x|\ge t \\ y\in 16B}}
\lf(\frac{t}{t+|x-y|}\r)^{\lz n}\lf|\int_{|y-z|<t}\frac{\Omega(y-z)}{|y-z|^{n-\rho}}b(z)\,dz\r|^2\,\frac{dydt}{t^{n+2\rho+1}}\r]^{\frac12} \\
&\hs+\lf[\int\int_{\substack{|y-x|\ge t \\ y\in (16B)^\complement \\ t\le|y-x_0|+8r}}\cdot\cdot\cdot\r]^{\frac12}
+\lf[\int\int_{\substack{|y-x|\ge t \\ y\in (16B)^\complement \\ t>|y-x_0|+8r}}\cdot\cdot\cdot\r]^{\frac12}=:{\rm{J_1+J_2+J_3}}.
\end{align*}

For ${\rm{J_1}}$, by $x\in (64B)^\complement$, $y\in 16B$ and $z\in B$, we know that
\begin{align*}
|y-x|\ge|x-x_0|-|y-x_0|>|x-x_0|-\frac14|x-x_0|>\frac12|x-x_0| \ {\rm{and}} \ |y-z|<32r.
\end{align*}

From this, Minkowski's inequality for integrals, $\bz<\min\{(\lz-2)n/3,\,\rho-n/2\}$,
$\Omega\in L^\fz(S^{n-1})$, and $|y-x|\sim|x-x_0|$ with $x\in (64B)^\complement$ and $y\in 16B$,
it follows that, for any $x\in (64B)^\complement$,
\begin{align*}
{\rm{J_1}}
&\le\lf[\int\int_{\substack{|y-x|\ge t \\ y\in 16B \\ |y-x|>|x-x_0|/2}}\lf(\frac{t}{t+|x-y|}\r)^{\lz n}
\lf|\int_{\substack{|y-z|<t \\ |y-z|<32r}}\frac{\Omega(y-z)}{|y-z|^{n-\rho}}b(z)\,dz\r|^2\,\frac{dydt}{t^{n+2\rho+1}}\r]^{\frac12} \\
&\le \int_B|b(z)|\lf[\int\int_{\substack{|y-x|\ge t \\ y\in 16B \\ |y-x|>|x-x_0|/2 \\ |y-z|<32r,\,|y-z|<t}}
\lf(\frac{t}{t+|x-y|}\r)^{\lz n}\frac{|\Omega(y-z)|^2}{|y-z|^{2n-2\rho}}\frac{dydt}{t^{n+2\rho+1}}\r]^{\frac12}dz \\
&\ls \int_B|b(z)|\lf[\int\int_{\substack{y\in 16B \\ |y-x|\ge t \\ |y-x|>|x-x_0|/2 \\ |y-z|<32r,\,|y-z|<t}}\lf(\frac{t}{|y-z|}\r)^{2\rho-n-2\bz}\r. \\
&\lf. \hspace{5.8 cm} \times \lf(\frac{t}{|x-x_0|}\r)^{2n+3\bz}\frac{|\Omega(y-z)|^2}{|y-z|^{2n-2\rho}}\frac{dydt}{t^{n+2\rho+1}}\r]^{\frac12}dz \\
&\ls \int_B|b(z)|\lf[\int_{|y-z|<32r}\frac{1}{|x-x_0|^{2n+3\bz}|y-z|^{n-2\bz}}\lf(\int_0^{|y-x|}t^{\bz-1}dt\r)dy\r]^{\frac12}dz \\
&\sim \int_B|b(z)|\lf(\int_{|y-z|<32r}\frac{|y-x|^\bz}{|x-x_0|^{2n+3\bz}|y-z|^{n-2\bz}}dy\r)^{\frac12}dz \\
&\sim \frac{1}{|x-x_0|^{n+\bz}}\int_B|b(z)|\lf(\int_{|y|<32r}\frac{1}{|y|^{n-2\bz}}dy\r)^{\frac12}dz \\
&\ls \|b\|_{L^\fz}\frac{r^n}{|x-x_0|^{n+\bz}}\lf(\int_{S^{n-1}}\int^{32r}_0\frac{1}{u^{n-2\bz}}u^{n-1}\,dud\sigma(y')\r)^\frac12 \\
&\sim \|b\|_{L^\fz}\frac{r^{n+\bz}}{|x-x_0|^{n+\bz}},
\end{align*}
which is wished.

For ${\rm{J_2}}$, write
\begin{align*}
{\rm{J_2}}
&\le\lf[\int\int_{\substack{|y-x|\ge t \\ y\in (16B)^\complement \\ t\le|y-x_0|+8r \\ |x-x_0|>3|y-x_0|}}
\lf(\frac{t}{t+|x-y|}\r)^{\lz n}\lf|\int_{|y-z|<t}\frac{\Omega(y-z)}{|y-z|^{n-\rho}}b(z)\,dz\r|^2\,\frac{dydt}{t^{n+2\rho+1}}\r]^{\frac12} \\
&\hs+\lf[\int\int_{\substack{|y-x|\ge t \\ y\in (16B)^\complement \\ t\le|y-x_0|+8r \\ |x-x_0|\le3|y-x_0|}}
\lf|\int_{|y-z|<t}\frac{\Omega(y-z)}{|y-z|^{n-\rho}}b(z)\,dz\r|^2\,\frac{dydt}{t^{n+2\rho+1}}\r]^{\frac12}=:{\rm{J_{21}+J_{22}}}.
\end{align*}

The estimate of ${\rm{J_{22}}}$ is quite similar to that given earlier for the estimate of ${\rm{I_{2}}}$ in Lemma \ref{m2} and so is omitted.
We are now turning to the estimate of ${\rm{J_{21}}}$.

For ${\rm{J_{21}}}$, by $x\in (64B)^\complement$, $y\in (16B)^\complement$, $z\in B$, $|x-x_0|>3|y-x_0|$ and the mean value theorem, we know that
\begin{align}\label{p1}
|y-z|\sim|y-x_0|;
\end{align}
\begin{align}\label{p2}
|y-x_0|-2r\le|y-x_0|-|x_0-z|\le|y-z|<t\le|y-x_0|+8r;
\end{align}
\begin{align}\label{p4}
|x-y|\ge|x-x_0|-|y-x_0|> \frac12|x-x_0|;
\end{align}
\begin{align}\label{p5}
\lf|\frac{1}{(|y-x_0|-2r)^{2\rho-n-2\bz}}-\frac{1}{(|y-x_0|+8r)^{2\rho-n-2\bz}}\r|\ls\frac{r}{|y-x_0|^{2\rho-n-2\bz+1}}.
\end{align}

From Minkowski's inequality for integrals, $\bz<\min\{(\lz-2)n/3,\,1/2\}<(\lz-2)n/2$, \eqref{p1}-\eqref{p5} and $\Omega\in L^\fz(S^{n-1})$, it follows that, for any $x\in (64B)^\complement$,
\begin{align*}
{\rm{J_{21}}}
&=\lf[\int\int_{\substack{|y-x|\ge t \\ y\in (16B)^\complement \\ t\le|y-x_0|+8r \\ |x-x_0|>3|y-x_0|}}
\lf(\frac{t}{t+|x-y|}\r)^{\lz n}\lf|\int_{|y-z|<t}\frac{\Omega(y-z)}{|y-z|^{n-\rho}}b(z)\,dz\r|^2\,\frac{dydt}{t^{n+2\rho+1}}\r]^{\frac12} \\
&\le \int_B|b(z)|\lf[\int\int_{\substack{|y-x|\ge t \\ y\in (16B)^\complement \\ t\le|y-x_0|+8r \\ |x-x_0|>3|y-x_0| \\ |y-z|<t }}
\lf(\frac{t}{t+|x-y|}\r)^{2n+2\bz}\frac{|\Omega(y-z)|^2}{|y-z|^{2n-2\rho}}\frac{dydt}{t^{n+2\rho+1}}\r]^\frac12dz \\
&\ls \int_B|b(z)|\lf[\int\int_{\substack{ y\in (16B)^\complement \\ |x-y|>|x-x_0|/2 \\ |y-x_0|-2r\le t\le|y-x_0|+8r }}
\lf(\frac{t}{|x-y|}\r)^{2n+2\bz}\frac{|\Omega(y-z)|^2}{|y-x_0|^{2n-2\rho}}\frac{dydt}{t^{n+2\rho+1}}\r]^\frac12dz \\
&\ls \frac{1}{|x-x_0|^{n+\bz}}\int_B|b(z)|\lf[\int_{\substack{ y\in (16B)^\complement}}
\frac{|\Omega(y-z)|^2}{|y-x_0|^{2n-2\rho}}\lf(\int^{|y-x_0|+8r}_{|y-x_0|-2r}\frac{dt}{t^{2\rho-n-2\bz+1}}\r)dy\r]^\frac12dz \\
&\ls \|b\|_{L^\fz}\frac{r^{n}}{|x-x_0|^{n+\bz}}
\lf(\int_{B^\complement}\frac{1}{|y-x_0|^{2n-2\rho}}\frac{r}{|y-x_0|^{2\rho-n-2\bz+1}}dy\r)^\frac12 \\
&\sim \|b\|_{L^\fz}\frac{r^{n+1/2}}{|x-x_0|^{n+\bz}}\lf(\int_{S^{n-1}}\int^\fz_r\frac{1}{u^{n-2\bz+1}}u^{n-1}\,dud\sigma(y')\r)^\frac12 \\
&\sim \|b\|_{L^\fz}\frac{r^{n+1/2}}{|x-x_0|^{n+\bz}}\,r^{\bz-1/2}\sim \|b\|_{L^\fz}\frac{r^{n+\bz}}{|x-x_0|^{n+\bz}},
\end{align*}
which is also wished.


For ${\rm{J_3}}$, noticing that $t>|y-x_0|+8r$, we see that, for any $y\in (16B)^\complement$,
\begin{align}\label{p6}
B\subset \{z\in\rn: \ |z-y|<t\}
\end{align}
and
\begin{align}\label{p7}
t+|x-y|\ge t+|x-x_0|-|y-x_0|\ge |x-x_0|+8r>|x-x_0|.
\end{align}
From \eqref{p6}, the vanishing moments of $b$, Minkowski's inequality for integrals, \eqref{p7},
\eqref{q6} and $\bz<\min\{\az,\,\rho-n/2,\,(\lz-2)n/3\}<(\lz-2)n/2$,
it follows that, for any $x\in (64B)^\complement$,
\begin{align*}
{\rm{J_{3}}}
&= \lf[\int\int_{\substack{|y-x|\ge t \\ y\in (16B)^\complement \\ t>|y-x_0|+8r}}
\lf(\frac{t}{t+|x-y|}\r)^{\lz n}\lf|\int_{|y-z|<t}\frac{\Omega(y-z)}{|y-z|^{n-\rho}}b(z)\,dz\r|^2\,\frac{dydt}{t^{n+2\rho+1}}\r]^{\frac12} \\
&\le\int_B |b(z)|
\lf[\int\int_{\substack{|y-x|\ge t \\ y\in (16B)^\complement \\ t>|y-x_0|+8r \\ |y-z|<t \\ t+|x-y|>|x-x_0|}}
\lf(\frac{t}{t+|x-y|}\r)^{\lz n}
\lf|\frac{\Omega(y-z)}{|y-z|^{n-\rho}}-\frac{\Omega(y-x_0)}{|y-x_0|^{n-\rho}}\r|^2\frac{dydt}{t^{n+2\rho+1}}\r]^\frac12dz \\
&\ls\int_B |b(z)|
\lf[\int\int_{\substack{|y-x|\ge t \\ y\in (16B)^\complement \\ t>|y-x_0|+8r \\ |y-z|<t \\ t+|x-y|>|x-x_0|}}
\lf(\frac{t+|x-y|}{|x-x_0|}\r)^{2n+2\bz}\lf(\frac{t}{t+|x-y|}\r)^{\lz n} \r. \\
&\lf. \hspace{5.4 cm}\times\frac{|z-x_0|^{2\az}}{|y-x_0|^{2n-2\rho+2\az}}\frac{dydt}{t^{n+2\rho+1}}\r]^\frac12dz \\
&\ls\frac{r^\az}{|x-x_0|^{n+\bz}}\int_B |b(z)|
\lf[\int\int_{\substack{|y-x|\ge t \\ y\in (16B)^\complement \\ t>|y-x_0|+8r \\ |y-z|<t \\ t+|x-y|>|x-x_0|}}
\frac{t^{\lz n}}{(t+|x-y|)^{\lz n-2n-2\bz}} \r.\\
&\lf. \hspace{7.3 cm} \times\frac{1}{|y-x_0|^{2n-2\rho+2\az}}\frac{dydt}{t^{n+2\rho+1}}\r]^\frac12dz \\
&\ls\frac{r^\az}{|x-x_0|^{n+\bz}}\int_B |b(z)|
\lf(\int\int_{\substack{ y\in B^\complement \\ t>|y-x_0|}}
\frac{1}{|y-x_0|^{2n-2\rho+2\az}}\frac{dydt}{t^{-n+2\rho-2\bz+1}}\r)^\frac12dz \\
&\ls\|b\|_{L^\fz}\frac{r^{n+\az}}{|x-x_0|^{n+\bz}}
\lf[\int_{B^\complement}\frac{1}{|y-x_0|^{2n-2\rho+2\az}}\lf(\int_{|y-x_0|}^\fz
\frac{dt}{t^{-n+2\rho-2\bz+1}}\r)dy\r]^\frac12  \\
&\sim\|b\|_{L^\fz}\frac{r^{n+\az}}{|x-x_0|^{n+\bz}}
\lf(\int_{B^\complement}\frac{1}{|y-x_0|^{n+2\az-2\bz}}dy\r)^\frac12 \\
&\sim \|b\|_{L^\fz}\frac{r^{n+\az}}{|x-x_0|^{n+\bz}}\lf(\int_{S^{n-1}}\int^\fz_r\frac{1}{s^{n+2\az-2\bz}}s^{n-1}\,dsd\sigma(y')\r)^\frac12 \\
&\sim \|b\|_{L^\fz}\frac{r^{n+\az}}{|x-x_0|^{n+\bz}}\,r^{\bz-\az}\sim\|b\|_{L^\fz}\frac{r^{n+\bz}}{|x-x_0|^{n+\bz}}.
\end{align*}

Combining the estimates of ${\rm{J_1}}$, ${\rm{J_{2}}}$ and ${\rm{J_3}}$, we obtain the desired inequality.
This finishes the proof of Lemma \ref{m4}.
\end{proof}

%

\begin{proof}[Proof of Theorem \ref{dingli.4}]
Obviously, $\mu^{\rho,\,\ast}_{\Omega,\,\lz}$ is a positive sublinear operator and bounded on $L^2$.
Thus, by the boundedness criterions of operators from $H^\vz$ to $WL^\vz$ (see Lemma \ref{yt2}),
Theorem \ref{dingli.4} will be proved by showing that $\mu^{\rho,\,\ast}_{\Omega,\,\lz}$
maps all multiple of $(\vz,\,\fz,\,s)$-atoms into uniformly bounded elements of $WL^\vz$, namely,
there exists a positive constant $C$ such that,
for any $\eta\in(0,\,\fz)$ and multiple of a $(\vz,\,\fz,\,s)$-atom $b$
associated with some ball $B(x_0,\,r)\subset\rn$,
$$\sup_{\alpha\in(0,\,\fz)}\vz\lf(\lf\{\mu^{\rho,\,\ast}_{\Omega,\,\lz}(b)>\alpha\r\},\,\frac{\alpha}{\eta}\r)
\le C\vz\lf(B,\,\frac{\|b\|_{L^\fz}}{\eta}\r).$$

For any $\eta\in(0,\,\fz)$, write
\begin{align*}
&\sup_{\alpha\in(0,\,\fz)}\vz\lf(\lf\{\mu^{\rho,\,\ast}_{\Omega,\,\lz}(b)>\alpha\r\},\,\frac{\alpha}{\eta}\r) \\
&\hs\le \sup_{\alpha\in(0,\,\fz)}\vz\lf(\lf\{x\in 64B: \ \mu^{\rho,\,\ast}_{\Omega,\,\lz}(b)(x)>\alpha\r\},\,\frac{\alpha}{\eta}\r) \\
&\hs\hs+\sup_{\alpha\in(0,\,\fz)}\vz\lf(\lf\{x\in (64B)^\complement: \ \mu^{\rho,\,\ast}_{\Omega,\,\lz}(b)(x)>\alpha\r\},\,\frac{\alpha}{\eta}\r)
=:{\rm{I_1}}+{\rm{I_2}}.
\end{align*}

For ${\rm{I_1}}$, by the uniformly upper type 1 property of $\vz$,
the boundedness on $L^2_{\vz(\cdot,\,t)}$, uniformly in $t\in(0,\,\fz)$, of $\mu^{\rho,\,\ast}_{\Omega,\,\lz}$ with $\vz\in\aa_2$
(see \cite[Theorem 1 and Remark 2]{xd07}),
and Lemma \ref{bcd}(i) with $\vz\in\aa_2$, it follows that, for any $\eta\in(0,\,\fz)$,
\begin{align*}
{\rm{I_1}}
&=\sup_{\alpha\in(0,\,\fz)}\int_{\lf\{x\in 64B: \ \mu^{\rho,\,\ast}_{\Omega,\,\lz}(b)(x)>\alpha\r\}}\vz\lf(x,\,\frac{\alpha}{\eta}\r)\,dx \\
&\le \int_{64B}\vz\lf(x,\,\frac{\mu^{\rho,\,\ast}_{\Omega,\,\lz}(b)(x)}{\eta}\r)\,dx \\
&\ls \int_{64B}\lf(1+\frac{|\mu^{\rho,\,\ast}_{\Omega,\,\lz}(b)(x)|}{\|b\|_{L^\fz}}\r)^2\vz\lf(x,\,\frac{\|b\|_{L^\fz}}{\eta}\r)\,dx \\
&\ls \int_{64B}\lf(1+\frac{|\mu^{\rho,\,\ast}_{\Omega,\,\lz}(b)(x)|^2}{\|b\|^2_{L^\fz}}\r)\vz\lf(x,\,\frac{\|b\|_{L^\fz}}{\eta}\r)\,dx \\
&\ls \vz\lf(64B,\,\frac{\|b\|_{L^\fz}}{\eta}\r)+\frac{1}{\|b\|^2_{L^\fz}}\int_\rn |\mu^{\rho,\,\ast}_{\Omega,\,\lz}(b)(x)|^2 \vz\lf(x,\,\frac{\|b\|_{L^\fz}}{\eta}\r)\,dx \\
&\ls \vz\lf(64B,\,\frac{\|b\|_{L^\fz}}{\eta}\r)+\frac{1}{\|b\|^2_{L^\fz}}\int_{B} |b(x)|^2 \vz\lf(x,\,\frac{\|b\|_{L^\fz}}{\eta}\r)\,dx \\
&\ls \vz\lf(B,\,\frac{\|b\|_{L^\fz}}{\eta}\r).
\end{align*}

For ${\rm{I_2}}$, from Lemma \ref{m4}, Lemma \ref{bcd}(i) with $\vz\in\aa_1$,
and the uniformly lower type $p$ property of $\vz$, we deduce that, for any $\eta\in(0,\,\fz)$,
\begin{align*}
{\rm{I_2}}
&\ls\sup_{\alpha\in(0,\,\fz)}\vz\lf(\lf\{x\in (64B)^\complement: \
\|b\|_{L^\fz}\frac{r^{n+\bz}}{|x-x_0|^{n+\bz}}>\alpha\r\},\,\frac{\alpha}{\eta}\r) \\
&\ls\sup_{\alpha\in(0,\,\fz)}\vz\lf(\lf\{x\in B^\complement: \
|x-x_0|^{n+\bz}<\frac{\|b\|_{L^\fz}}{\alpha}r^{n+\bz}\r\},\,\frac{\alpha}{\eta}\r) \\
&\sim\sup_{\alpha\in(0,\,\fz)}\vz\lf(\lf\{x\in \rn: \
r\le|x-x_0|<\lf(\frac{\|b\|_{L^\fz}}{\alpha}\r)^{\frac{1}{n+\bz}} r \r\},\,\frac{\alpha}{\eta}\r) \\
&\ls\sup_{\alpha\in(0,\,\|b\|_{L^\fz})}\vz\lf(\lf\{x\in \rn: \
|x-x_0|<\lf(\frac{\|b\|_{L^\fz}}{\alpha}\r)^{\frac{1}{n+\bz}} r \r\},\,\frac{\alpha}{\eta}\r) \\
&\thicksim\sup_{\alpha\in(0,\,\|b\|_{L^\fz})}
\vz\lf(\lf[\frac{\|b\|_{L^\fz}}{\alpha}\r]^{\frac{1}{n+\bz}} B ,\,\frac{\alpha}{\eta}\r) \\
&\ls\sup_{\alpha\in(0,\,\|b\|_{L^\fz})} \lf(\frac{\|b\|_{L^\fz}}{\alpha}\r)^p
\vz\lf(B ,\,\frac{\alpha}{\eta}\r) \\
&\ls\sup_{\alpha\in(0,\,\|b\|_{L^\fz})} \lf(\frac{\|b\|_{L^\fz}}{\alpha}\r)^p
\lf(\frac{\alpha}{\|b\|_{L^\fz}}\r)^p \vz\lf(B,\,\frac{\|b\|_{L^\fz}}{\eta}\r) \\
&\thicksim \vz\lf(B,\,\frac{\|b\|_{L^\fz}}{\eta}\r).
\end{align*}

To summarize what we have proved, we obtain the desired inequality.
This finishes the proof of Theorem \ref{dingli.4}.
\end{proof}

\noindent {\bf Acknowledgements}

\medskip

\noindent
The author is greatly indebted to Associate Professor Li Baode for many useful discussions and for the guidance over the past years.

\medskip

\noindent Li Bo

\medskip
%
%

\noindent{E-mail }:
\texttt{bli.math@outlook.com}   \\

\bigskip

\end{document}